\documentclass[11pt,leqno]{article}

\usepackage{amssymb,amsmath,epsfig,amsthm,rotating}

\newcommand{\sst}{\scriptstyle}
\newcommand{\n}{\noindent}
\newcommand{\bb}[1]{\mathbb{#1}}
\newcommand{\cl}[1]{\mathcal{#1}}

\newcommand{\ovl}{\overline}
\newcommand{\intl}{\int\limits}

\theoremstyle{plain}
\newtheorem{lem}{Lemma}[section]

\newtheorem{cor}{Corollary}[section]
\newtheorem{prop}{Proposition}[section]

\theoremstyle{remark}

\newtheorem{rk}{Remark}[section]

\theoremstyle{definition}
\newtheorem{exm}{Example}[section]

\numberwithin{equation}{section}

\newif\ifproofmode
\proofmodetrue % change true to false to erase the notes

\font\eightrm=cmr6

\def\note#1{%
  \ifproofmode%
    \vadjust{%
      \setbox1=\vtop{%
        \hsize 1cm\parindent=0pt\eightrm\baselineskip=9pt%
        \rightskip=2mm plus 2mm\raggedright#1%
        }%
      \hbox{\kern-1.2cm\smash{\box1}\hfil}%
      }%
  \else
\relax
\fi}

\begin{document}

\title{Bi-Isometries and Commutant Lifting}

\author{Hari Bercovici\footnote{Research partially supported by grants from the National Science Foundation}, Ronald G.~Douglas$^*$ and Ciprian Foias}

\date{}
\maketitle

\begin{center}
\textit{\Large In memory of M.S.~Livsic,}\\
\textit{\Large  one of the founders of modern operator theory}
\end{center}\bigskip

\begin{abstract}
In a previous paper, the authors obtained a model for a bi-isometry,
that is, a pair of commuting isometries on complex Hilbert space.
This representation is based on the canonical model of Sz.~Nagy and
the third author. One approach to describing the invariant subspaces
for such a bi-isometry using this model is to consider isometric
intertwining maps from another such model to the given one.
Representing such maps requires a careful study of the commutant
lifting theorem and its refinements. Various conditions relating to
the existence of isometric liftings are obtained in this note, along
with some examples demonstrating the limitations of our results.
\end{abstract}

\vspace{1.5in}

\n \textit{2000 Mathematics Subject Classification}:\ 46G15, 47A15, 47A20, 47A45, 47B345.

\n \textit{Key Words and Phrases}:\ Bi-isometries, commuting isometries, canonical model, Commutant lifting, intertwining maps, invariant subspaces.
\newpage

\section{Introduction}\label{cfsec1}

\indent

The geometry of complex Hilbert space is especially transparent. In particular, all Hilbert spaces of the same dimension are isometrically isomorphic. One consequence is the simple structure of isometric operators on Hilbert space as was discovered by von~Neumann  in his study of symmetric operators in connection with quantum mechanics. A decade later, this decomposition was rediscovered by Wold who made it the basis for his study of stationary stochastic processes. Another decade later, Beurling obtained his iconic result on invariant subspaces for the unilateral shift operator. While his proof did not rely on the structure of isometries later works showed that the result could be established using it.
In the fifties, Sz.-Nagy demonstrated that all contraction operators on Hilbert space had a unique
 minimal unitary dilation. The application of structure theory for isometries to this unitary
 operator is one starting point for the canonical model theory of Sz.-Nagy and the third author \cite{SzNF2}.
 Much of the development of this theory, including the lifting theorem for intertwining
 operators and the parametrization of the possible lifts can be viewed as exploiting and refining
 the structure theory of
 isometric operators on complex Hilbert space.

The study of commuting $n$-tuples of isometries is not so simple,
even for $n=2$. This paper makes a contribution to this theory. The
starting point is the model introduced implicitly in \cite{BDF} for
a bi-isometry or a pair of commuting isometries. We now describe the
model explicitly. Let $\{\Theta(z), {\cl E},{\cl E}\}$
 be a contractive operator-valued analytic function $(z\in {\bb D})$ and
 set $\Delta(\zeta) = (I - \Theta(\zeta)^* \Theta(\zeta))^*$, $\zeta\in \partial{\bb D}$.
 Define the Hilbert space
\begin{equation}\label{cfeq1.1}
 {\cl H}_\Theta = H^2({\cl E}) \oplus H^2(\ovl{\Delta L^2({\cl E})})
\end{equation}
and the operators
\begin{equation}\label{cfeq1.1a}
V_\Theta(f\oplus g) = f_1\oplus g_1, W_\Theta(f\oplus g) = f_2\oplus g_2,\tag{1.1a}
\end{equation}
where
\begin{align}\label{cfeq1.1b}
 &f_1(z) = zf(z), \quad f_2(z) = \Theta(z)f(z)\qquad (z\in {\bb D})\tag{1.1b}\\
\label{cfeq1.1c}
&g_1(w,\zeta) = \zeta g(w,\zeta), g_2(w,\zeta) = \Delta(\zeta) f(\zeta) + wg(w,\zeta) \qquad (w\in {\bb D}, \zeta\in \partial {\bb D}).\tag{1.1c}
\end{align}
Then $(V_\Theta,W_\Theta)$ is a bi-isometry such that there is no nonzero reducing subspace for $(V_\Theta, W_\Theta)$ on which
 $V_\Theta$ is unitary.

In \cite{BDF2} we have shown
 that \emph{any bi-isometry $(V,W)$, for which there is no nonzero reducing
 subspace ${\cl N}$ such that $V|{\cl N}$ is unitarily
  equivalent to a bi-isometry $(V_\Theta,W_\Theta)$, where $\Theta(\cdot)$ is
  uniquely determined up to coincidence}. (Note that the terminology and the
  notations are as
   in \cite{SzNF2}.)

An important part of the study of this model is a description of all invariant subspaces of the bi-isometry $(V_\Theta,W_\Theta)$. To this end we first describe all the contractive operators $Y$ intertwining two bi-isometries $(V_{\Theta_1},W_{\Theta_1})$ and $(V_\Theta,W_\Theta)$; that is, $Y\in {\cl L}({\cl H}_{\Theta_1},{\cl H}_\Theta)$ and
\begin{equation}\label{cfeq1.2}
 YV_{\Theta_1} = V_\Theta Y,\quad YW_{\Theta_1} = W_\Theta Y.
\end{equation}
Let $P$ denote the orthogonal projection of ${\cl H}_\Theta$ onto $H^2({\cl E}) (\approx H^2({\cl E}) \oplus \{0\} \subset {\cl H}_\Theta)$. Then there exists a unique
 contractive analytic operator-valued function $\{A(\cdot), {\cl E}_1, {\cl E}\}$ such that
\[
 (PYh_1)(z) = A(z)h_1(z)\qquad (z\in {\bb D})
\]
for all $h_1\in H^2({\cl E}_1) (\approx H^2({\cl E}_1)\oplus \{0\} = {\cl H}_{\Theta_1})$. Conversely, given such a contractive analytic function $A(\cdot)$, there exists a contractive intertwining operator $Y$, but it is not unique. Using the Commutant Lifting Theorem, one can describe completely the set of such intertwining contractions.  The description involves an analytic operator-valued function $\{R(\cdot), {\cl R}, {\cl R}'\}$, called the free Schur contraction in Section \ref{cfsec2}. Here, the spaces $\cl R$ and $\cl R'$  are called residual spaces, and they are entirely determined by the functions $\Theta_1,\Theta$ and $A$.
 If ${\cl M}$ is a common invariant subspace for the bi-isometry $(V_\Theta,W_\Theta)$,
 then defining $U_1 = V_\Theta|{\cl M}$ and $U_2 = W_\Theta|{\cl M}$
  yields a bi-isometry $(U_1,U_2)$ on ${\cl M}$. Moreover, the inclusion
  map $X\colon \ {\cl M}\to {\cl H}_\Theta$ is an isometric intertwining map.
  Conversely, if $Y$ is an isometric intertwining map from a
  model $(V_{\Theta_1}, W_{\Theta_1})$ on ${\cl H}_{\Theta_1}$ to ${\cl H}_{\Theta}$,
  then the range of $Y$ is a common invariant subspace for the
  bi-isometry $(V_\Theta,W_\Theta)$. Hence, the problem of describing the common
  invariant subspaces for $(V_\Theta,W_\Theta)$ is closely related to describing the
  isometric intertwining maps from some
  model $(V_{\Theta_1}, W_{\Theta_1})$ to $(V_\Theta,W_\Theta)$.

Thus the description of all the invariant subspaces of
$(V_\Theta,W_\Theta)$ is intimately connected to the determination
of the class of the free Schur contractions for which the
corresponding operator $Y$ is an isometry. As yet we have not found
a completely satisfactory characterization of that set. In this Note
we present our contributions to
 this problem with the hope that they may be instrumental in the discovery of
 an \emph{easily applicable} characterization.

In the next section we provide a description of the commutant lifting theorem focusing on the aspects relevant to our problem. In Section \ref{cfsec3} the analytical details are  taken up while in the fourth section we state our results on isometric intertwining maps. In the final section, we apply these results to
 the question of invariant subspaces in those cases in which our results are effective.
  We conclude
 with a number of open questions and future directions for study.

\section{A short review of the Commutant Lifting Theorem}\label{cfsec2}

\indent

Let $T' \in {\cl L}({\cl H}')$ be a completely nonunitary (c.n.u.) contraction and $T\in {\cl L}({\cl H})$ an isometry; here ${\cl H}, {\cl H}'$ are (separable) Hilbert spaces. Furthermore, let $X\in {\cl L}({\cl H}, {\cl H}')$ be a contraction intertwining $T$ and $T'$; that
 is,
\begin{equation}\label{cfeq2.1}
 T'X = XT,\quad\|X\|\le1.
\end{equation}
Let $U'\in {\cl L}({\cl K}')$ be a minimal isometric lifting of $T'$.  In other words, if $P'$ denotes the orthogonal projection  ${\cl K}'$ onto ${\cl H}'$ and $I'$ denotes
 the identity operator on ${\cl K}'$, we have
\begin{equation}\label{cfeq2.2}
P'U' = T'P',
\end{equation}

\begin{align}\label{cfeq2.2a}
 U^{\prime *}U' &= I',\tag{2.2a}\\
\intertext{and}
\label{cfeq2.2b}
{\cl K}' &= \bigvee^\infty_{n=0} U^{\prime n} {\cl H}'.\tag{2.2b}
\end{align}
Since ${\cl H}'$  is essentially unique,
 one can take
\begin{align}\label{cf2.2c}
 &{\cl K}' = {\cl H}' \oplus H^2({\cl D}_{T'})\tag{2.2c}\\
\label{cfeq2.2d}
&U'(h' \oplus f(\cdot)) = T'h' \oplus (D_{T'}h' + \cdot f(\cdot))\tag{2.2d}
\end{align}
for all $h'\in {\cl H}', f(\cdot)\in H^2({\cl D}_{T'})$. Recall that
$D_{T'}=(I-T'T^{\prime*})^{1/2}$ and
 ${\cl D}_{T'}=(D_{T'}{\cl H}')^-$.

In its original form \cite{sarason,SzNF1}, \emph{the Commutant Lifting Theorem asserts that there exists an operator $X\in {\cl L}({\cl H}, {\cl K}')$ satisfying the following properties}
\begin{align}\label{cfeq2.3}
 &U'Y = YT,\\
\label{cfeq2.3a}
&\|Y\|\le 1,\tag{2.3a}\\
\label{cfeq2.3b}
&P'Y= X.\tag{2.3b}
\end{align}
Such an operator $Y$ is called a \emph{contractive intertwining
lifting of} $X$. In this study we need a tractable classification of
all these liftings. To this aim we introduce the  \emph{isometry}
\begin{equation}
 \omega:(D_XT{\cl H})^-\to\{D_{T'}Xh\oplus D_Xh\colon \ h\in {\cl H}\}^-,
\end{equation}
 obtained by closing the linear operator
\begin{equation}\label{cfeq2.4a}
 \omega_0\colon \ D_XTh\mapsto D_{T'} Xh\oplus D_Xh\qquad (h\in {\cl H});\tag{2.4a}
\end{equation}
and the partial isometry operator $\bar\omega\in{\cl L}({\cl D}_X, {\cl D}_{T'})$ defined by
\begin{equation}\label{cfeq2.4b}
 \bar\omega|(D_XT{\cl H})^- = \omega,\quad \bar\omega|({\cl D}_X \ominus (D_XT{\cl H})^-) = 0.\tag{2.4b}
\end{equation}
 This operator obviously satisfies
\begin{equation}\label{cfeq2.4c}
 \ker \bar\omega = {\cl D}_X \ominus(D_XT{\cl H})^-, \quad \ker \bar\omega^* = ({\cl D}_{T'} \oplus {\cl D}_X) \ominus \text{ran } \omega.\tag{2.4c}
\end{equation}
Also we will denote
 by $\Pi$ and $\Pi'$ the operators on ${\cl D}_{T'}\oplus {\cl D}_X$ defined by
\begin{equation}\label{cfeq2.4d}
 \Pi'(d'\oplus d) = d', \quad \Pi(d'\oplus d) = d\qquad (d'\in {\cl D}_{T'}, d\in {\cl D}_X).\tag{2.4d}
\end{equation}
With
this preparation we can state the needed
description (see \cite{FFGK}, Ch. VI).

\begin{prop}\label{cfprop2.1}
 {\rm (i)}~~Any contractive intertwining lifting $Y$ of $X$ is of the form
\begin{equation}\label{cfeq2.5}
 Y = \begin{bmatrix}
      X\\ \Gamma(\cdot) D_X
     \end{bmatrix}\colon \ {\cl H}\mapsto \underset{\sst H^2({\cl D}_{T'})}{\overset{\sst{\cl H}}{\textstyle\bigoplus}}
 = {\cl K}',
\end{equation}
where $\Gamma(\cdot)$ is given by the formula
\begin{equation}\label{cfeq2.5a}
 \Gamma(z) = \Pi'W(z) (1-z\Pi W(z))^{-1}\qquad (z\in {\bb D}),\tag{2.5a}
\end{equation}
where
\begin{equation}\label{cfeq2.5b}
 W(z)\colon \ {\cl D}_X \mapsto {\cl D}_{T'}\oplus {\cl D}_X\qquad (z\in {\bb D})\tag{2.5b}
\end{equation}
is a contractive analytic function satisfying
\begin{equation}\label{cfeq2.5c}
 W(z)|(D_XT{\cl H})^- = \omega\qquad (z\in {\bb D}).\tag{2.5c}
\end{equation}

{\rm (ii)}~~Conversely, for any $W(\cdot)$ satisfying the above conditions, the formulas \eqref{cfeq2.5} and \eqref{cfeq2.5a} yield a contractive intertwining lifting $Y$ of $X$.

{\rm (iii)}~~The correspondence between  $Y$ and $W(\cdot)$ is one-to-one.
\end{prop}

The function $W(\cdot)$ is called the \emph{Schur contraction} of
$Y$, and $Y$ is called the contractive intertwining lifting
associated to $W(\cdot)$.
 It is immediate that the Schur contraction is uniquely determined by its restriction
\begin{equation}\label{cfeq2.6}
R(z)= W(z)|\ker\bar\omega:\ker\bar\omega\to\ker\bar\omega^*\qquad (z\in {\bb D}).
\end{equation}
Any contractive analytic function $\{R(\cdot),\ker\bar\omega,\ker\bar\omega^*\}$ determines a Schur contraction.  The function $R(\cdot)$ will be called the \emph{free Schur contraction} of $Y$. Thus we have the following

\begin{cor}\label{cfcor2.2}
 The formulas \eqref{cfeq2.5}, \eqref{cfeq2.5a}, \eqref{cfeq2.6} establish a bijection between the set of all contractive intertwining liftings of $X$ and the set of all free Schur contractions.
\end{cor}

As already stated, the main purpose of this paper is to study the free Schur contractions for which the associated contractive intertwining lifting is isometric; in particular, to find necessary conditions on $T,T'$ and $X$ for such an isometric lifting to exist.

\section{Analytic considerations}\label{cfsec3}

\indent

Let ${\cl D}, {\cl D}'$ be two (separable) Hilbert spaces and let
\begin{equation}\label{cfeq3.1}
W(z) = \begin{bmatrix}
A(z)\\ B(z)
       \end{bmatrix} \in {\cl L}({\cl D}, {\cl D}\oplus {\cl D}'), \left\|
\begin{bmatrix}
A(z)\\ B(z)
\end{bmatrix}\right\|\le 1\qquad (z\in {\bb D})
\end{equation}
be analytic in ${\bb D}$. Define an analytic function $\Gamma$ by setting
\begin{equation}\label{cfeq3.2}
\Gamma(z) = B(z) (I-zA(z))^{-1}\in{\cl L}({\cl D}, {\cl D}') \qquad (z\in{\bb D}).
\end{equation}

\begin{lem}\label{cflem3.1}
For all $d\in {\cl D}$ the function $\Gamma d$ defined by
 $\Gamma d(z)=\Gamma(z)d$ belongs to $\in H^2({\cl D}')$, and
\begin{align}\label{cfeq3.3}
\|\Gamma d\|^2_{H^2({\cl D}')} &= \lim_{\rho\nearrow 1} \frac1{2\pi} \int^{2\pi}_0 \|\Gamma(\rho e^{i\theta})d\|^2d\theta\\
&= \|d\|^2 - \lim_{\rho\nearrow 1} \left[\left(\frac1{\rho^2}-1\right) \frac1{2\pi} \int^{2\pi}_0 \|(I-\rho e^{i\theta} A(\rho e^{i\theta}))^{-1}d\|^2 d\theta\right.\notag\\
&\quad  \left. + \frac1{2\pi} \int^{2\pi}_0 \|D_{W(\rho e^{i\theta})} (I-\rho e^{i\theta} A(\rho e^{i\theta}))^{-1} d\|^2d\theta\right].\notag
\end{align}
\end{lem}

\begin{proof}
For $\rho\in (0,1)$, the function $(1-\rho zA(\rho z))^{-1}d$  is
bounded, and in particular it belongs to $H^2({\cl D})$. Moreover,
it can be decomposed as a sum of two orthogonal vectors in $H^2({\cl
D})$ as follows:
$$(1-\rho zA(\rho z))^{-1}d=d+\rho zA(\rho z)(1-\rho zA(\rho z))^{-1}d.$$
Thus we have
\begin{align*}
\frac1{2\pi} \int^{2\pi}_0 \|\Gamma(\rho e^{i\theta})d\|^2 d\theta &= \frac1{2\pi} \int^{2\pi}_0 \|W(\rho e^{i\theta}) (I-\rho e^{i\theta} A(\rho e^{i\theta}))^{-1} d\|^2 d\theta\\
&\quad - \frac1{2\pi} \int^{2\pi}_0 \|A(\rho e^{i\theta}) (I-\rho e^{i\theta} A(\rho e^{i\theta}))^{-1} d\|^2 d\theta\\
&= \frac1{2\pi} \int^{2\pi}_0 \|W(\rho e^{i\theta}) (I-\rho e^{i\theta}A(\rho e^{i\theta}))^{-1}d\|^2 dt \\
&\quad - \frac1{2\pi} \int^{2\pi}_0 \frac1{\rho^2} \|-d + (I-\rho e^{i\theta} A(\rho e^{i\theta}))^{-1} d\|^2 d\theta \\
\intertext{\newpage}
&= \frac1{2\pi} \int^{2\pi}_0 \|W(\rho e^{i\theta}) (I-\rho e^{i\theta} A(\rho e^{i\theta}))^{-1} d\|^2 + \frac1{\rho^2} \|d\|^2 \\
&\quad - \frac1{\rho^2} \frac1{2\pi} \int^{2\pi}_0 \|(I-\rho e^{i\theta} A(\rho e^{i\theta}))^{-1}d\|^2 d\theta \\
&= \frac1{\rho^2} \|d\|^2 - \frac1{2\pi} \int^{2\pi}_0 \|D_{W(\rho e^{i\theta})} (I-\rho e^{i\theta} A(\rho e^{i\theta})^{-1}d\|^2 d\theta \\
&\quad - \left(\frac1{\rho^2} - 1\right) \frac1{2\pi} \int^{2\pi}_0 \|(I-\rho e^{i\theta} A(e^{i\theta}))^{-1} d\|^2 d\theta + \left(\frac1{\rho^2} -1\right) \|d\|^2,
\end{align*}
from which \eqref{cfeq3.3} follows by letting $\rho\nearrow 1$.
\end{proof}

The following equality is an obvious consequence of Lemma \ref{cflem3.1}.

\begin{cor}\label{cfcor3.1}
The map $\Gamma(\cdot)\colon \ d(\in {\cl D})\mapsto \Gamma d$ is a
contraction from ${\cl D}$ into $H^2({\cl D}')$ and
\begin{align}\label{cfeq3.3a}
\|D_{\Gamma(\cdot)}d\|^2 &= \lim_{\rho\nearrow 1} \left[\frac1{2\pi} \int^{2\pi}_0 \|D_{W(\rho e^{i\theta})} (I-\rho e^{i\theta} A(\rho e^{i\theta}))^{-1}d\|^2 d\theta\right.\tag{3.3\text{a}}\\
&\quad \left. + \left(\frac1{\rho^2}-1\right) \frac1{2\pi} \int^{2\pi}_0 \|(I-\rho e^{i\theta} A(\rho e^{i\theta}))^{-1} d\|^2 d\theta\right]\qquad (d\in{\cl D}).\notag
\end{align}
\end{cor}

\begin{lem}\label{cflem3.2}
For all $d\in {\cl D}$ we have
\begin{align}\label{cfeq3.4}
&\lim_{\rho\nearrow 1} \left[\|d\|^2 - \frac1{2\pi} \int^{2\pi}_0 \|D_{A(\rho e^{i\theta})} (I-\rho e^{i\theta} A(\rho e^{i\theta}))^{-1}d\|^2d\theta\right]\\
&\quad = \lim_{\rho\nearrow 1} \left(\frac1{\rho^2}-1\right) \frac1{2\pi} \int^{2\pi}_0 \|(I-\rho e^{i\theta} A(\rho e^{i\theta}))^{-1}d\|^2d\theta.\notag
\end{align}
\end{lem}

\begin{proof}
For $d\in {\cl D}$ we have
\begin{align*}
&\frac1{2\pi} \int^{2\pi}_0 \|D_{A(\rho e^{i\theta})} (I-\rho e^{i\theta}A(\rho e^{i\theta}))^{-1}d\|^2 d\theta\\
&\quad = \frac1{2\pi} \int^{2\pi}_0 \left[\|(I-\rho e^{i\theta}A(\rho e^{i\theta}))^{-1}d\|^2 - \frac1{\rho^2} \|-d + (I- \rho e^{i\theta}A(\rho e^{i\theta}))^{-1}d\|^2\right]d\theta\\
&\quad= \frac1{\rho^2} \|d\|^2 + \left(1-\frac1{\rho^2}\right) \frac1{2\pi}\int^{2\pi}_0 \|(I-\rho e^{i\theta}A(\rho e^{i\theta}))^{-1}d\|^2 d\theta,
\end{align*}
from which \eqref{cfeq3.4} readily follows.
\end{proof}

\begin{lem}\label{cflem3.3}
 Let $d\in {\cl D}$ and set
\begin{equation}\label{cfeq3.5}
d(z) = (I-zA(z))^{-1}d = d_0 + zd_1 +\cdots+ z^nd_n+\cdots\quad (z\in {\bb D}),
\end{equation}
where $d_n\in {\cl D}$ and $d_0=d$. If
\begin{equation}\label{cfeq3.5a}
\|d_n\|\to 0\quad \text{for}\quad n\to\infty,\tag{3.5\text{a}}
\end{equation}
then we also have
\begin{equation}\label{cfeq3.5b}
I_\rho = \left(\frac1{\rho^2}-1\right) \frac1{2\pi} \int^{2\pi}_0 \|(I-\rho e^{i\theta}A(\rho e^{i\theta}))^{-1}d\|^2 d\theta\to 0 \quad \text{for}\quad \rho\nearrow 1.\tag{3.5\text{b}}
\end{equation}
\end{lem}

\begin{proof}
Observe that
\[
I_\rho = \frac{1+\rho}{\rho^2} (1-\rho) \sum^\infty_{n=0} \rho^{2n} \|d_n\|^2,
\]
so for any $N=1,2,\ldots$ we have
\[
\limsup_{\rho\nearrow 1} I_\rho \le 2 \limsup_{\rho\nearrow 1} (1-\rho) \sum^\infty_{n=N} \rho^{2n} \|d_n\|^2 \le \max \{\|d_n\|^2\colon \ n\ge N\},
\]
which is assumed to tends to 0 as $N\to\infty$.
\end{proof}

\begin{lem}\label{cflem3.4}
Let $d$ and $d(z)(z\in {\bb D})$ be as in Lemma \ref{cflem3.3}. Then the equality
 $\|\Gamma d\|^2_{H^2({\cl D}')} = \|d\|^2$ implies the convergence in \eqref{cfeq3.5a}.
\end{lem}

\begin{proof}
Let
\[
 \Gamma(z)d = g_0 + zg_1 +\cdots+ z^ng_n +\cdots\qquad (z\in {\bb D}),
\]
where $g_n\in {\cl D}'$ and note that
\[
 W(z)d(z) =
\begin{bmatrix}
d_1 + zd_2+z^2d_2+\cdots\\
g_0 + zg_1 + z^2g_2+\cdots
\end{bmatrix}\qquad (z\in {\bb D}).
\]
Thus for $n=2,3,\ldots$, we have
\begin{align*}
&\left\|\begin{bmatrix}
         d_1+z d_2 +\cdots+ z^{n-1}d_n\\
g_0 + z g_1 +\cdots+ z^{n-1}g_{n-1}
        \end{bmatrix}\right\|^2_{H^2}  = \frac1{2\pi} \int^{2\pi}_0 \left\|
\begin{bmatrix}
d_1 + e^{i\theta}d_2 +\cdots + e^{i(n-1)\theta}d_{n}\\
g_0 + e^{i\theta}g_2 +\cdots+ e^{i(n-1)\theta}g_{n-1}
\end{bmatrix}\right\|^2 d\theta\\
&\quad = \|d_1\|^2 + \|d_2\|^2 +\cdots+ \|d_n\|^2 + \|g_0\|^2 +\cdots+ \|g_{n-1}\|^2\\
&\quad \le \frac1{2\pi} \int^{2\pi}_0 \|W(e^{i\theta}) (d_0+e^{i\theta} d_1 +\cdots+ e^{i(n-1)\theta}d_{n-1})\|^2 d\theta\\
&\quad\le \frac1{2\pi} \int^{2\pi}_0 \|d_0 + e^{i\theta}d_1 +\cdots+ e^{i(n-1)}d_{n-1}\|^2 d\theta\\
&\quad = \|d_0\|^2 + \|d_1\|^2 +\cdots+ \|d_{n-1}\|^2,
\end{align*}
where $d_0=d$. Thus we obtain
\[
 \|d_n\|^2 + \|g_0\|^2 +\cdots+ \|g_{n-1}\|^2 \le \|d\|^2.
\]
But since $\|\Gamma d\|^2_{H^2} = \|g_0\|^2 +\cdots+ \|g_{n-1}\|^2
+\cdots$, the above inequality and the assumption that
$\|\Gamma(\cdot)d\|_{H^2} = \|d\|$ implies \eqref{cfeq3.5a}.
\end{proof}

We can now state and prove the main result of this section.

\begin{prop}\label{cfprop3.1}
The following sets of properties {\rm (a), (b), (c)} and {\rm (d)} are equivalent:
\begin{itemize}
 \item[\rm (a)] $\Gamma(\cdot)$ is an isometry;
\item[\rm (b)] $W(\cdot)$ and $A(\cdot)$ satisfy the conditions
\end{itemize}
\begin{align}\label{cfeq3.6}
 &\lim_{\rho\nearrow 1} \int^{2\pi}_0 \|D_{W(\rho e^{i\theta})} (I-\rho e^{i\theta} A(\rho e^{i\theta}))^{-1}d\|^2 = 0\qquad (d\in {\cl D})\\
\intertext{and}
\label{cfeq3.6a}
&\lim_{\rho\nearrow 1} (1-\rho) \int^{2\pi}_0 \|(I-\rho e^{i\theta}A(\rho e^{i\theta})^{-1}d\|^2 = 0\qquad (d\in {\cl D});\tag{3.6\text{a}}
\end{align}
\begin{itemize}
 \item[\rm (c)] $W(\cdot)$ and $A(\cdot)$ satisfy the condition \eqref{cfeq3.6} and
\end{itemize}
\begin{equation}\label{cfeq3.6b}
\lim_{\rho\nearrow 1} \frac1{2\pi} \int^{2\pi}_0 \|D_{A(\rho e^{i\theta})} (I-\rho e^{i\theta}A(\rho e^{i\theta}))^{-1}d\|^2 d\theta = \|d\|^2\qquad (d\in {\cl D}); \text{ and}\tag{3.6\text{b}}
\end{equation}

\begin{itemize}
 \item[\rm (d)] $W(\cdot)$ and $A(\cdot)$
 satisfy \eqref{cfeq3.6}, and in the Taylor
  expansion
\[
(1-zA(z))^{-1} = I + zD_1 +\cdots+ z^nD_n+\cdots\qquad (z\in {\bb D})
\]
we have $D_n\to 0$ strongly {\rm(}that is, $\lim_{n\to\infty} \|D_nd\|= 0$  for all $d\in {\cl D}${\rm)}.
\end{itemize}
\end{prop}

\begin{proof}
The equivalence of the properties (a) and (b) follows directly from the Corollary \ref{cfcor3.1}. This corollary and Lemma \ref{cflem3.4} show that the property (d) implies the property (a). The converse implication follows readily from the same corollary and Lemma \ref{cflem3.3}. Finally, Lemma \ref{cflem3.2} shows that the conditions \eqref{cfeq3.6a} and \eqref{cfeq3.6b} are equivalent and thus so are the sets of properties (b) and (c).
\end{proof}

\begin{cor}\label{cfcor3.2}
Assume that $W(z_0) = 0$ for some $z_0\in {\bb D}$. Then $\Gamma(\cdot)$ is an isometry if and only if \eqref{cfeq3.6} is valid.
\end{cor}

\begin{proof}
We have
\[
\|W(z)\| \le |z-z_0|/|1-\bar z_0z|\qquad (z\in {\bb D})
\]
and consequently, for all $d\in {\cl D}$, we also have
\[
\|D_{W(z)}d\|^2 \ge (1-|z-z_0|^2/|1-\bar z_0z|^2)\|d\|^2\ge ((1-|z_0|^2)/(1+|z_0z|)^2) (1-|z|^2) \|d\|^2\quad (z\in {\bb D}).
\]
It follows that
\[
\int^{2\pi}_0 \|D_{W(\rho e^{i\theta})} (I-\rho e^{i\theta}A(\rho e^{i\theta}))^{-1}d\|^2 \ge \frac{(1-|z_0|^2)(1+\rho)}{(1+|z_0|\rho)^2} (1-\rho) \int^{2\pi}_0 \|(I-\rho e^{i\theta} A(\rho e^{i\theta}))^{-1} d\|^2 d\theta
\]
for all $d\in {\cl D}$. Thus \eqref{cfeq3.6} implies \eqref{cfeq3.6a}.
\end{proof}

Another special case of Proposition \ref{cfprop3.1} is given by the following.

\begin{cor}\label{cfcor3.3}
Assume that $W(z) \equiv W_0\in{\cl L}({\cl D},{\cl D}\oplus{\cl D}')$ $(z\in {\bb D})$. Then $\Gamma(\cdot)$ is an isometry if and only if
\begin{itemize}
\item[\rm (d$'$)] $W_0$ is an isometry and $A_0 = A(0) (=A(z)(z\in {\bb D}))$ is a $C_{0\bullet}$-contraction; that is,
\end{itemize}
\begin{equation}\label{cfeq3.7}
\|A^n_0d\|\to 0\quad \text{for}\quad n\to\infty\qquad (d\in {\cl D}).
\end{equation}
\end{cor}

\begin{proof}
In this case $D_{W(z)}(1-zA(z))^{-1}d = D_{W_0}(1-zA(z))^{-1}d$ (as a ${\cl D}$-valued function of $z$) belongs to $H^2({\cl D})$ for any $d\in {\cl D}$. Therefore, \eqref{cfeq3.6} implies that $D_{W_0}(1-zA(z))^{-1}d = 0$ $(z\in {\bb D})$ and, in particular, $D_{W_0}=0$; that is, $W_0$ is an isometry. Clearly, if this last property holds for  $W_0$, then \eqref{cfeq3.6} is trivially true. According to the equivalence of the properties (a) and (d) in Proposition \ref{cfprop3.1} and the fact that in the present case $D_n = A^n_0$ $(n=0,1,\ldots)$, the property (d$'$) above is equivalent to (a).
\end{proof}

\begin{rk}\label{cfrk3.1}
In Corollary \ref{cfcor3.3}, the condition \eqref{cfeq3.7} is not superfluous. Indeed, if $A_0$ is any contraction for which \eqref{cfeq3.7} fails,
 then  define $W = W_0 = \left[\begin{smallmatrix} C\\ D_C\end{smallmatrix}\right]$.
 This $W$ is an isometry but $\Gamma(\cdot)$ is not isometric. Thus, in general,
 the (actually equivalent) conditions \eqref{cfeq3.6a} and \eqref{cfeq3.6b} are not superfluous.
\end{rk}

Proposition \ref{cfprop3.1} has an interesting connection to the
Herglotz representation (cf. \cite[p. 3]{duren}) of an analytic
operator-valued function
\[
F(z)\in {\cl L}({\cl H}),\qquad z\in {\bb D}
\]
(where ${\cl H}$ is a Hilbert space) such that
\begin{equation}\label{cfeq3.8}
 \text{Re } F(z) (= (F(z) + F(z)^*)/2) \ge 0, \quad \text{Im } F(0) = \left(\frac{F(0)-F(0)^*}2\right)= 0.
\end{equation}
This representation is
\begin{equation}\label{cfeq3.8a}
 F(z) = \intl_{\partial{\bb D}} \frac{\zeta+z}{\zeta-z} E(d\zeta)\qquad (z\in {\bb D}),\tag{3.8a}
\end{equation}
where $E(\cdot)$ is a  positive operator-valued measure on $\partial{\bb D} = \{\zeta := |\zeta| = 1\}$, uniquely determined by $F$
 (an early occurrence of this representation is in \cite[Theorem 3]{livsic} ).

To explicate that connection, we first observe that the function
\begin{equation}\label{cfeq3.8b}
F(z) = (I + zA(z)) (I-zA(z))^{-1}\qquad (z\in {\bb D}),\tag{3.8b}
\end{equation}
satisfies the inequality \eqref{cfeq3.8}. Indeed, we have (with $d(z) = (1-zA(z))^{-1}d, z\in {\bb D}$)
\begin{align}\label{cfeq3.8c}
h_d(z) := ((\text{Re } F(z))d,d) &= \text{Re}(F(z)d,d) = \text{Re}((I + zA(z)) d(z), (I-zA(z)) d(z)) =\tag{3.8c}\\
&= \|d(z)\|^2 - \|zAd(z)\|^2 = \|D_{zA(z)}d(z)\|^2\ge 0.\notag
\end{align}
Thus  our particular $F(\cdot)$ has a representation of the form
 \eqref{cfeq3.8a}. Now from \eqref{cfeq3.8a}
we easily infer
\begin{equation}\label{cfeq3.8d}
h_d(z) = \intl_{\partial{\bb D}} \text{Re } \frac{\zeta+z}{\zeta-z} (E(d\zeta)d,d)\qquad (z\in {\bb D}).\tag{3.8d}
\end{equation}
Therefore
\begin{align}\label{cfeq3.8e}
\frac1{2\pi} \intl^{2\pi}_0 h_d(\rho e^{i\theta}) d\theta  &= \intl_{\partial D} \left(\frac1{2\pi} \intl^{2\pi}_0 \text{Re } \frac{\zeta+\rho e^{i\theta}}{\zeta-\rho e^{i\theta}} d\theta\right) (E(d\zeta)d,d)=\tag{3.8e}\\
&= \intl_{\partial {\bb D}} (E(d\zeta)d,d) = \|d\|^2 \quad \text{(for $\rho\in (0,1)$).}\notag
\end{align}
Since $h_d(z)$ is a nonnegative harmonic function in ${\bb D}$, the
limit
\begin{equation}\label{cfeq3.8f}
 \lim_{\rho\nearrow 1} h_d(\rho e^{i\theta}) = h_d(e^{i\theta}) \tag{3.8f}
\end{equation}
exists a.e.\ in $[0,2\pi)$, and the absolutely continuous part of the measure $(E(\cdot)d,d)$ has
 density equal to $h_d(e^{i\theta})/2\pi$ a.e.\ (cf. \cite[p.5-6]{duren} or  \cite[Chapter 2]{hoffman})

It follows that the singular part $\mu_d(\cdot)$ of $(E(\cdot)d,d)$ satisfies
\begin{equation}\label{cfeq3.8g}
 \mu_d(\partial{\bb D}) = \|d\|^2 - \intl^{2\pi}_0 h_d(e^{i\theta}) \frac{d\theta}{2\pi}.\tag{3.8g}
\end{equation}
Consequently, the  following facts (a) and (b) are equivalent:
\begin{itemize}
 \item[(a)] the measure $E(\cdot)$ is absolutely continuous; and
\item[(b)] the relation
\end{itemize}
\begin{equation}\label{cfeq3.8h}
 \|d\|^2 = \intl^{2\pi}_0 h_d(e^{i\theta}) \frac{d\theta}{2\pi}\tag{3.8h}
\end{equation}
\begin{itemize}
\item[] holds for all $d\in {\cl D}$.
\end{itemize}

We recall that the function $\Gamma d(z)=\Gamma(z)d$ is in $H^2({\cl
D}')$ for every $d\in{\cl D}$.  In particular, this function has
radial limits a.e. on $\partial\bb D$.

\begin{lem}\label{cflem3.5}
Let $d\in {\cl D}$ be fixed, and define
\begin{equation}\label{cfeq3.8i}
 k_d(z) = \frac{1-|z|^2}{|z|^2} \|d(z)\|^2 + \frac1{|z|^2} \|D_{W(z)}d(z)\|^2\qquad (z\in {\bb D}\backslash \{0\}),\tag{3.8i}
\end{equation}
where
\begin{equation}\label{cfeq3.8ia}
 d(z) = (1-zA(z))^{-1}d\qquad (z\in {\bb D}).\tag{3.8ia}
\end{equation}
Then
\begin{equation}\label{cfeq3.8j}
 \lim_{\varphi\nearrow 1} k_d(\rho e^{i\theta}) := k_d(e^{i\theta})\tag{3.8j}
\end{equation}
exists a.e.\ and
\begin{equation}\label{cfeq3.8k}
 k_d(e^{i\theta}) + \|\Gamma d(e^{i\theta})\|^2 = h_d(e^{i\theta})\quad \text{a.e.}\tag{3.8k}
\end{equation}
\end{lem}

\begin{proof}
We first observe that
\begin{equation}\label{cfeq3.8l}
 k_d(z) +\|\Gamma(z)d\|^2 = \frac1{|z|^2} h_d(z)\qquad (z\in {\bb D}\backslash\{0\}).\tag{3.8l}
\end{equation}
Indeed, for $z\in {\bb D}, z\ne 0$ we have
\begin{align*}
\|D_{W(z)}d(z)\|^2 + \|\Gamma(z)d\|^2 &= \|D_{A(z)}d\|^2 = \|d(z)\|^2 \\
&\quad - \frac1{|z|^2} \|zA(z)d(z)\|^2 = \left(1-\frac1{|z|^2}\right) \|d(z)\|^2 \\
&\quad + \frac1{|z|^2} \|D_{zA(z)}d(z)\|^2 = \left(1- \frac1{|z|^2}\right) \|d(z)\|^2 + \frac1{|z|^2} h_d(z).
\end{align*}
Thus \eqref{cfeq3.8f} and \eqref{cfeq3.8l} imply \eqref{cfeq3.8j}
 and \eqref{cfeq3.8k}.
\end{proof}

We can now give the following complement to Proposition
\ref{cfprop3.1} which establishes the connection between the
absolute continuity of the measure $E(\cdot)$ in \eqref{cfeq3.8a}
and the fact that $\Gamma$ is an isometry when viewed as an operator
from ${\cl D}$ to $H^2({\cl D}')$.

\begin{prop}\label{cfprop3.2}
The operator $\Gamma\in {\cl L}({\cl D}, H^2({\cl D}'))$ is an
isometry if and only if the following two conditions are satisfied:
\begin{align}\label{cfeq3.8m}
&\text{The measure $E(\cdot)$ in the representation \eqref{cfeq3.8a} of the function $F(\cdot)$ defined in \eqref{cfeq3.8b}} \tag{3.8m}\\
&\text{is absolutely continuous, and}\notag
\end{align}
\begin{equation}\label{cfeq3.8n}
 k_d(e^{i\theta}) = 0\quad\text{a.e.}\qquad (d\in {\cl D}),\tag{3.8n}
\end{equation}
where $k_d(\cdot)$ is defined in \eqref{cfeq3.8i}, \eqref{cfeq3.8j}.
\end{prop}

\begin{proof}
Assume first that $\Gamma$ is an isometry. Then from
\eqref{cfeq3.8k} we infer (for any $d\in {\cl D}$)
\begin{equation}\label{cfeq3.8o}
 \frac1{2\pi} \int^{2\pi}_0 k_d(e^{i\theta})d\theta + \|d\|^2 = \frac1{2\pi} \int^{2pi}_0 h_d(e^{i\theta}) d\theta \le \|d\|^2.\tag{3.8o}
\end{equation}
Since $k_d(e^{i\theta})\ge 0$ a.e., \eqref{cfeq3.8o} implies \eqref{cfeq3.8n} and \eqref{cfeq3.8h}. Consequently, (due to the equivalence of the facts (a) and (b), above; see the discussion preceding Lemma \ref{cflem3.5}), we have that \eqref{cfeq3.8m} is also valid. Conversely, assume that both statements \eqref{cfeq3.8m} and \eqref{cfeq3.8n} are valid. Then using again \eqref{cfeq3.8k}  we have
\[
 \|\Gamma d\|^2_{H^2({\cl D}')} = \frac1{2\pi} \int^{2\pi}_0 \|(\Gamma d) (e^{i\theta})\|^2 d\theta = \frac1{2\pi} \int^{2\pi}_0 h_d(e^{i\theta})d\theta = \|d\|^2
\]
for all $d\in {\cl D}$. Consequently, $\Gamma$ is an isometry.
\end{proof}

For the investigation of the basic condition \eqref{cfeq3.6} we need to study first the operator-valued functions
\begin{align}\label{cfeq3.9}
K(z) &= 1-zA(z)\qquad (z\in {\bb D})\\
\intertext{and}
\label{cfeq3.9a}
J(z) &= K(z)^{-1}\qquad (z\in {\bb D}).\tag{3.9\text{a}}
\end{align}

\begin{lem}\label{cflem3.6}
{\rm a)}~~$K(z)$ $(z\in {\bb D})$ is an outer function; and\\
{\rm b)}~~$\|K(z)\|\le 1$ $(z\in {\bb D})$ if and only if $A(z)=0$ $(z\in {\bb D})$.
\end{lem}

\begin{proof}
  For $n=1,2,3,\ldots$ we have
\[
 I - (zA(z))^n = (1-zA(z)) (1+zA(z) +\cdots+ (zA(z))^{n-1})\qquad (z\in {\bb D}).
\]
Consequently, for any function $h\in H^2({\cl D})$ the function
\[
 d(z) - z^nA(z)^n d(z)
\]
belongs to the range ${\cl R}$ of the operator of multiplication by
$(1-z A(z))$ on $ H^2({\cl D})$. But, since $\|A(z)^n h(z)\|\le
\|h(z)\|$ the functions
\[
 z^n A(z)^nh(z)
\]
converge to zero weakly in $ H^2({\cl D})$. It follows that ${\cl
R}$ is weakly dense in $H^2({\cl D})$ and hence (being a subspace of
$H^2({\cl D})$) also strongly dense in $H^2({\cl D})$. This proves
the first part of the lemma. For the second part, assume that
$\|K(z)\|\le 1$ for all $z\in\bb D$. Then

\[
 \|K(z)d\|\le \|d\| = \|K(0)d\|,\quad d\in{\cl D},
\]
and the maximum principle forces $K(z)d = K(0)d=d$ $(d\in {\cl D})$.
Thus $A(z)=0$ for $(z\in {\bb D})$.
\end{proof}

\begin{rk}\label{cfrk3.2}
Proposition \ref{cfprop3.1} in the case $A(z) = 0$ $(z\in {\bb D})$ takes the following trivial form:\ $\Gamma(\cdot)$ is an isometry if and only if $W(\cdot)$ is inner. Therefore, from now on we will assume that $A(z)\not\equiv 0$, or equivalently that
\begin{equation}\label{cfeq3.9b}
\underset{\scriptstyle |\zeta|=1}{\text{ess sup}}
\|K(\zeta)\|>1.\tag{3.9\text{b}}
\end{equation}
\end{rk}

\begin{rk}\label{cfrk3.3}
It is worth noticing that the basic condition \eqref{cfeq3.6} implies that if $F(z) \in {\cl L}({\cl D}, {\cl D}'')$ $(z\in {\bb D})$  is a bounded operator-valued analytic function, where ${\cl D}''$ is any Hilbert space, such that
\[
 F(e^{it})^* F(e^{it}) \le D^2_{W(e^{it})} = I_{\cl D} - W(e^{it})^* W(e^{it}) \text{ a.e.,}
\]
then $F(z) \equiv 0$ $(z\in {\bb D})$. Indeed, for the bounded analytic function
\[
G(z) = \begin{bmatrix}
        W(z)\\ F(z)
       \end{bmatrix}\qquad (z\in {\bb D}),
\]
we have
\[
 G(e^{i\theta})^* G(e^{i\theta}) \le I_{\cl D}\quad \text{a.e.}
\]
Therefore,
\begin{align}
G(z)^* G(z) &\le I_{\cl D}\qquad (z\in {\bb D}),\notag\\
F(z)^*F(z) &\le D^2_{W(z)} \qquad (z\in {\bb D}),\notag\\
\label{cfeq3.9c}
\|F(z) J(z)d\|^2 &\le \|D_{W(z)}J(z)d\|^2\qquad (z\in {\bb D}, d\in {\cl D}), \tag{3.9c}
\end{align}
and hence by virtue of \eqref{cfeq3.6}
\[
\lim_{\rho\nearrow 1} \int^{2\pi}_0 \|F(\rho e^{i\theta}) J(\rho e^{i\theta})d\|^2 d\theta \le \lim_{\rho \nearrow 1} \int^{2\pi}_0 \|D_{W(\rho e^{i\theta})} J(e^{i\theta})d\|^2 d\theta = 0.
\]
It follows that the ${\cl D}''$-valued function $F(z)J(z)$ (in $H^2({\cl D}'')$) is identically 0. Thus
\[
F(z) = F(z) \cdot J(z) K(z)  \equiv 0\qquad (z\in {\bb D}).
\]
Note that by virtue of (\cite[p. 201--203]{SzNF2}), the result we just established is equivalent to
\begin{equation}\label{cfeq3.9z}
\overline{D_{W(\cdot)}H^2({\cl D})}^{L^2({\cl D})} = \overline{D_{W(\cdot)}L^2({\cl D})}^{L^2({\cl D})},
\end{equation}
where both closures are in $L^2({\cl D})$. Thus \eqref{cfeq3.9z} is a necessary condition for $\Gamma(\cdot)$ to be an isometry.
\end{rk}

It is obvious that if $W(\cdot)$ is inner (that is, ${D_{W(e^{i\theta})}} = 0$ a.e.), then \eqref{cfeq3.9z} is satisfied. We will give now a case in which  the basic condition \eqref{cfeq3.6} in Proposition \ref{cfprop3.1} can be replaced with the condition
\begin{equation}\label{cfeq3.10}
D_{W(e^{it})} = 0 \quad \text{(a.e.);}
\end{equation}
that is, $W(\cdot)$ is an inner (analytic) function. To this end we
recall that the analytic operator-valued function $(z\in {\bb D})$
$(\cdot)$ is said to have a scalar multiple if there exist a
$\delta(\cdot)\in H^\infty$ and a bounded operator-valued analytic
function $G(z)$ $(z\in {\bb D})$ such that
\begin{equation}\label{cfeq3.11}
K(z) G(z) = G(z)K(z)  = \delta(z) I_{\cl D}\not\equiv 0\qquad (z\in {\bb D})
\end{equation}
(cf. \cite[Ch.~V, Sec. 6]{SzNF2}). By adapting the proof of Theorem 6.2 (loc.\ cit.) to our situation, we can assume due to Lemma \ref{cflem3.6} a)  that $\delta$ is an outer function; that is,
\begin{equation}\label{cfeq3.11a}
 \ovl{\delta H^2} = H^2.\tag{3.12a}
\end{equation}
Consequently, so is $G(\cdot)$; that is,
\begin{equation}\label{cfeq3.11b}
\ovl{G(\cdot)H^2({\cl D})} = H^2({\cl D}),\tag{3.12b}
\end{equation}
and hence (cf.\  \cite[Ch.~V, Proposition 2.4 (ii)]{SzNF2})
\begin{equation}\label{cfeq3.11c}
 G(e^{it}){\cl D} = {\cl D} \quad \text{(a.e.)}\tag{3.12c}
\end{equation}
Note that \eqref{cfeq3.11}, \eqref{cfeq3.11a}, \eqref{cfeq3.11b},
and \eqref{cfeq3.11c} imply that
\begin{equation}\label{cfeq3.11d}
J(e^{it}) = K(e^{it})^{-1}  \text{ exists  in } {\cl L}({\cl D}) \text{ a.e.,}
\tag{3.12d}
\end{equation}
and is in fact equal to $G(e^{it})/\delta(e^{it})$ a.e.\ Moreover,
\begin{equation}\label{cfeq3.11e}
\|J(\rho e^{it})d - J(e^{it})d\| \to 0 \text{ for } \rho\nearrow 1 \text{ for all $d\in {\cl D}$, a.e. } \tag{3.12e}
\end{equation}

We will now consider the slightly more general case in which \eqref{cfeq3.11d}, \eqref{cfeq3.11a} hold regardless of whether a scalar multiplier exists for $K(\cdot)$.

\begin{lem}\label{cflem3.7}
 Assume that \eqref{cfeq3.11d} and \eqref{cfeq3.11e} hold. Then (see the notation in Proposition \ref{cfprop3.2})
\begin{align}\label{cfeq3.12}
k_d(e^{i\theta}) &= \|D_{W(e^{i\theta})} J(e^{i\theta})d\|^2 \quad \text{a.e.\ and}\\
\label{cfeq3.13a}
h_d(e^{i\theta}) &= \|D_{A(e^{i\theta})} J(e^{i\theta})d\|^2\quad \text{a.e.}\tag{3.13a}
\end{align}
\end{lem}

\begin{proof}
 Since $A(z)$ and $\tilde A(z) := A(\bar z)^*$ are analytic, we have
\[
A(\rho e^{i\theta}) \to A(e^{i\theta}),\qquad A(\rho e^{i\theta})^* \to A(e^{i\theta}) \quad \text{strongly a.e.}
\]
and consequently
\begin{align*}
 D_{A(\rho e^{i\theta})} &= (I-A(\rho e^{i\theta})^* A(\rho e^{i\theta}))^{1/2}\to\\
&\to (I-A(e^{i\theta})^* A(e^{i\theta}))^{1/2} = D_{A(e^{i\theta})} \quad \text{strongly a.e.}
\end{align*}
A similar argument holds for the strong convergence
\[
 D_{W(\rho e^{i\theta})} \to D_{W(e^{i\theta})}\quad \text{ a.e.}
\]
Relations \eqref{cfeq3.12} and \eqref{cfeq3.13a} are direct
consequences of the above strong convergences and of
\eqref{cfeq3.11e}.
\end{proof}

We can now give the following corollary to Proposition \ref{cfprop3.2}.

\begin{prop}\label{cfprop3.3}
Assume that  \eqref{cfeq3.11d} and \eqref{cfeq3.11e} hold. Then $\Gamma(\cdot) \in {\cl L}({\cl D}, H^2({\cl D}'))$ is an isometry if and only if the following property holds:
\begin{itemize}
 \item[(e)] $W(\cdot)$ satisfies the condition \eqref{cfeq3.10} and $A(\cdot)$ satisfies the condition
\end{itemize}
\begin{equation}\label{cfeq3.13}
\frac1{2\pi} \int^{2\pi}_0 \|D_{A(e^{i\theta})} (I-e^{i\theta}A(e^{i\theta}))^{-1}d\|^2 d\theta = \|d\|^2 \qquad (d\in {\cl D}).
\end{equation}
\end{prop}

\begin{proof}
In the present situation (due to Lemma \ref{cflem3.7}) we have that
$\Gamma(\cdot)$ is isometric if and only if  relation
\eqref{cfeq3.13} holds, and
\begin{equation}\label{cfeq3.14}
D_{W(e^{i\theta})} J(e^{i\theta})d = 0\quad \text{a.e.}\quad (d\in {\cl D})
\end{equation}
hold. Let ${\cl D}_0\subset {\cl D}$ be a countable dense subset of
${\cl D}$ and denote by $Ex(d)$ the null
 set on which  \eqref{cfeq3.14} fails. If $Ex_0$ denotes the set of the $e^{i\theta}$ is for
which at least one of the relations \eqref{cfeq3.11d},
\eqref{cfeq3.11e} is not valid, then
\[
 Ex = Ex_0 \cup \left(\bigcup_{d\in {\cl D}_0} Ex(d)\right)
\]
is also a null set and
\begin{equation}\label{cfeq3.14a}
D_{W(e^{i\theta})} J(e^{i\theta}){\cl D}_0 = \{0\}\qquad (e^{i\theta}\notin Ex).\tag{3.15a}
\end{equation}
But for $e^{i\theta}\notin Ex$, the operator $D_{W(e^{i\theta})} J(e^{i\theta})$ is bounded. Therefore \eqref{cfeq3.14a} implies
\[
 D_{W(e^{i\theta})} = D_{W(e^{i\theta})} J(e^{i\theta})\cdot  E(e^{i\theta}) = 0\cdot E(e^{i\theta}) = 0,
\]
i.e.\ \eqref{cfeq3.10}. This concludes the proof since \eqref{cfeq3.10} obviously implies \eqref{cfeq3.14}.
\end{proof}

\begin{cor}\label{cfcor3.4}
 Let $a,b\in H^\infty$ satisfy the condition
\[
 \left\|\begin{bmatrix}
         a(z)\\ b(z)
        \end{bmatrix}\right\| \le 1\qquad (z\in {\bb D})
\]
and define
\[
w(z) = \begin{bmatrix}
        a(z)\\ b(z)
       \end{bmatrix} \quad \text{and}\quad \gamma(z) = \frac{b(z)}{1-za(z)}\qquad (z\in {\bb D}).
\]
Then
\begin{equation}\label{cfeq3.16}
 \|\gamma(\cdot)\|^2 = \frac1{2\pi} \int^{2\pi}_0 |\gamma(e^{i\theta})|^2 d\theta \le 1
\end{equation}
and equality holds in \eqref{cfeq3.16} if and only if $w$ and $a$
 satisfy the following conditions:
\begin{gather}\label{cfeq3.16a}
w(e^{i\theta})^* w(e^{it}) = 1\quad\text{a.e.\quad and}\tag{3.17a}\\
\label{cfeq3.16b}
\frac1{2\pi} \int^{2\pi}_0 \frac{1-|a(e^{i\theta})|^2}{|1-e^{i\theta}a(e^{i\theta})|^2} d\theta =1. \tag{3.17b}
\end{gather}
\end{cor}

\begin{proof}
The result follows readily from Proposition \ref{cfprop3.2}, by taking ${\cl D} = {\bb C}$, ${\cl D}' = {\bb C}$.
\end{proof}

\begin{rk}\label{cfrk3.5}
In Proposition \ref{cfprop3.2}, neither one of the equalities \eqref{cfeq3.16} or \eqref{cfeq3.16b} implies the other, as is shown by the following two examples.
\end{rk}

\begin{exm}\label{cfexm3.1}
Define
\[
u(z) = \frac{3/4}{2\pi} \int^{\pi}_0 \frac{e^{i\theta}+z}{e^{i\theta}-z} d\theta + \frac{1/4}{2\pi} \int^{2\pi}_\pi \frac{e^{i\theta}+z}{e^{i\theta}-z} d\theta + \frac12\frac{1+z}{1-z} \quad (z\in {\bb D}).
\]
Then
\[
v(z) = \text{Re } u(z) \ge 0\qquad (z\in {\bb D})
\]
and
\[
v(e^{i\theta}) = \lim_{\rho\nearrow 1} v(\rho e^{i\theta}) =
\begin{cases}
3/4&\text{for $0<e^{i\theta} < \pi$}\\
1/4&\text{for $\pi < e^{i\theta} < 2\pi$}
\end{cases}~.
\]
Set
\[
a(z) = \frac1z \frac{u(z)-1}{u(z)+1}\quad (z\in {\bb D}\backslash\{0\})\quad \text{and}\quad a(0) = \frac12 u'(0).
\]
Then $a(z) \in H^\infty$ and
\begin{equation}\label{cfeq3.17}
1-|a(e^{i\theta})|^2 = \frac{4v(e^{i\theta})}{|u(e^{i\theta})+1|^2} = v(e^{i\theta}) |1-e^{i\theta} a(e^{i\theta})|^2 \text{ a.e.}
\end{equation}
In particular, this relation shows that $\|a(\cdot)\|_{H^\infty} \le 1$; moreover,
\[
1 - |a(e^{i\theta})|^2 \ge \frac14 |1-e^{i\theta}a(e^{i\theta})|^2 \quad\text{a.e.}
\]
Since $1-za(z) (z\in {\bb D})$ is an outer function, there exists  an outer function $b\in H^\infty$ such that
\[
|b(e^{i\theta})|^2 = 1 - |a(e^{i\theta})|^2 \quad \text{a.e.}
\]
Therefore
\[
\|w(z)\|\le 1\quad (z\in {\bb D}), \text{ where } w(z) =
\begin{bmatrix}
 a(z)\\ b(z)
\end{bmatrix}\quad (z\in {\bb D}),
\]
and \eqref{cfeq3.16a} is satisfied. However, due to \eqref{cfeq3.17}
\[
\frac1{2\pi} \int^{2\pi}_0 \frac{1-|a(e^{i\theta})|^2}{|1-e^{i\theta}a(e^{i\theta})|^2} d\theta = \frac1{2\pi} \int^{2\pi}_0 v(e^{i\theta})d\theta = \frac12,\quad \text{and}
\]
hence \eqref{cfeq3.16b} is not valid.
\end{exm}

\begin{exm}\label{cfexm3.2}
Define
\[
w(z) =\begin{bmatrix}
       1/2\\ 1/2
      \end{bmatrix} \quad (z\in {\bb D}).
\]
Then $\|w(z)\|\le 1/\sqrt 2 < 1$ $(z\in \ovl{\bb D})$ and \eqref{cfeq3.16} is not satisfied, but
\[
 \frac1{2\pi} \int^{2\pi}_0 \frac{1-(1/2)^2}{|1-e^{i\theta}/2|^2} d\theta = 1,
\]
i.e.\ \eqref{cfeq3.16b} is valid.
\end{exm}

\section{Isometric intertwining lifting}\label{cfsec4}

\indent

We return now to the commutant lifting theorem setting presented in Section \ref{cfsec3}. We recall that if we denote
\begin{equation}\label{cfeq4.1}
 A(z) = \Pi W(z),\quad d(z) = (I-zA(z))^{-1}d\qquad (z\in {\bb D}),
\end{equation}
where $d\in {\cl D} = {\cl D}_X$, then the contractive intertwining liftings of $X$ are given by the formula
\begin{equation}\label{cfeq4.1a}
 B = \begin{bmatrix}
      X\\ \Gamma(\cdot)D_X
     \end{bmatrix} \in {\cl L}({\cl H}, {\cl H}'\oplus H^2({\cl D})),\tag{4.1a}
\end{equation}
where
\begin{equation}\label{cfeq4.1b}
\Gamma(z)d = \Pi'W(z)d(z)\qquad (z\in {\bb D})\tag{4.1b}
\end{equation}
and
\begin{equation}\label{cfeq4.1c}
 W(z) = \bar\omega d(z) + R(z) (I-\bar\omega^*\bar\omega)d(z)\quad (z\in {\bb D}).\tag{4.1c}
\end{equation}
In \eqref{cfeq4.1c}, $\bar\omega$ is the partial isometry $\in {\cl
L}({\cl D}, {\cl D} \oplus {\cl D}_T)$ defined in
Section
\ref{cfsec2} and
\[
R(z) \in{\cl L}(\ker\bar\omega,\ker\bar\omega^*)\qquad (z\in {\bb
D})
\]
is an arbitrary analytic operator-valued function such that
\[
 \|R(z)\|\le 1\qquad (z\in {\bb D}).
\]
We also recall that $B$ is isometric if and only if $\Gamma(\cdot)$ is isometric. According to Proposition \ref{cfprop4.1} this can happen if and only if the (d) set of properties holds for $W(\cdot)$. By  noticing that
\begin{equation}\label{cfeq4.1d}
\begin{split}
\|D_{W(z)}d(z)\|^2 &= \|d(z)\|^2 - \|\bar\omega d(z)\|^2\notag\\
-\|R(z)(1-\bar\omega^*\bar\omega)d(z)\|^2 &= \|(1-\bar\omega^*\bar\omega)d(z)\|^2\notag\\
-\|R(z)(1-\bar\omega^*\bar\omega)d(z)\|^2 &= \|D_{R(z)}(1-\bar\omega^*\bar\omega) d(z)\|^2\qquad (z\in {\bb D}),
\end{split}\tag{4.1d}
\end{equation}
we have this result as a direct consequence of Proposition \ref{cfprop3.1}.

\begin{prop}\label{cfprop4.1}
 The contractive intertwining lifting $B$ associated to $R(\cdot)$ is isometric if and only if the following two properties hold for all $d\in {\cl D}$:
\begin{align}\label{cfeq4.2}
&\lim_{\varphi\nearrow 1} \frac1{2\pi} \int^{2\pi}_0 \|D_{R(\rho e^{i\theta})}(1-\bar\omega^* \bar\omega) d(e^{i\theta})\|^2 d\theta=0; \text{ and}\\
\label{cfeq4.2a}
&\text{the Taylor coefficients $d_n (n=0,1,\ldots)$ of $d(\cdot)$ satisfy the condition}\tag{4.2a}\\
&\|d_n\|\to 0 \text{ for } n\to \infty.\notag
\end{align}
\end{prop}

Due to \eqref{cfeq4.1d} (as well as to the equivalence of the facts (a), (b) observed before Lemma \ref{cflem3.5}) we can also reformulate Proposition \ref{cfprop3.2} as follows

\begin{prop}\label{cfprop4.2}
$\Gamma(\cdot)\in {\cl L}({\cl D}, H^2({\cl D}'))$ is an isometry if and only if the following two conditions are satisfied:
\begin{equation}\label{cfeq4.3}
 \int^{2\pi}_0 h_d(e^{i\theta}) \frac{d\theta}{2\pi} = \|d\|^2\qquad (d\in {\cl D}),
\end{equation}
where (see also \eqref{cfeq4.1})
\begin{equation}\label{cfeq4.3a}
 h_d(e^{i\theta}) = \lim_{\rho\to 1} \|D_{W(z)} d(z)\|\big|_{z=\rho e^{i\theta}} \text{ \rm a.e.}
\tag{4.3a}
\end{equation}
for each $d\in {\cl D}$; and
\begin{equation}\label{cfeq4.4}
 \lim_{\rho\nearrow 1} [(1-|z|^2) \|d(z)\|^2 + \|D_{R(z)} (1-\bar\omega^* \bar\omega) d(z)\|\big|_{z=\rho e^{i\theta}} = 0 \text{ \rm a.e.}
\end{equation}
for each $d\in {\cl D}$.
\end{prop}

The problem with these two propositions is that one cannot always
apply them.  None of the conditions \eqref{cfeq4.2},
\eqref{cfeq4.2a}, \eqref{cfeq4.3} or \eqref{cfeq4.4} is easy to
analyse or check. To illustrate this difficulty we will now give two
results.

\begin{prop}\label{cfprop4.3}
 With the notation of Section \ref{cfsec2}, assume $\|X\|<1$ and
 that there is an isometric intertwining lifting $Y_1$ of $X$.
 Then $T$ is a unilateral shift.
\end{prop}

\begin{proof}
 Let ${\cl H} = {\cl H}_0 \oplus {\cl H}_1$ be the Wold
 decomposition for $T$; that is, $T{\cl H}_0 \subset {\cl H}_0$,
  $T{\cl H}_1\subset {\cl H}_1$, $T|{\cl H}_0$ is a unilateral
  shift and $T|{\cl H}_1$ is unitary. Then
\[
 {\cl H}_1 = \bigcap^\infty_{n=0} T^n{\cl H}.
\]
Therefore, if $Y$ is any intertwining lifting for $X$ (that is,
\[
U'Y=YT, \quad P'Y=X),
\]
then we have
\[
 Y{\cl H}_1 \subseteq \bigcap^\infty_{n=0} U^{\prime n}{\cl K}' = {\cl R},
\]
where
\[
 {\cl K}' = {\cl R}^\bot \oplus {\cl R}
\]
is the Wold decomposition for $U'$. If $U'_1 = U'|{\cl R}$, where $Y_1$ is an isometric lifting of $X$, then $Z=Y-Y_1$ satisfies
\begin{align*}
U'_1(Z|{\cal H}_1) = Z(T|{\cl H}_1),\\
U^{\prime *}_1(Z|{\cl H}_1) = (Z|{\cl H}_1)(T|{\cal H}_1)^*
\end{align*}
since $U'_1$ and $T|{\cl H}_1$ are unitary. Therefore, $\ovl{Z{\cl H}_1}$ is a reducing subspace for $U'$ and is orthogonal to ${\cl H}'$ (because $P'Z = P'Y -P'Y_1 = 0$). Due to the minimality of $U'$ (that is,
\[
 {\cl K}' = \bigvee_{n>0} U^{\prime n}{\cl H}')
\]
we have $\ovl{Z{\cl H}_1} = \{0\}$ and hence
\begin{equation}\label{cfeq4.5}
 Y|{\cl H}_1 = Y_1|{\cl H}_1.
\end{equation}
But the Commutant Lifting Theorem applied to $X_0 = X/\|X\|$ yields a contractive intertwining lifting $Y_0$ of $X_0$. It follows that $Y := \|X\|Y_0$ is an intertwining lifting of $X$ such that $\|Y\|\le \|X\|<1$. From the relation \eqref{cfeq4.5} and the hypothesis that $Y_1$ is isometric, we conclude that ${\cl H}_1 = \{0\}$ and so $T=T_0$ is an unilateral shift.
\end{proof}

This result shows that if $X$ is a strict contraction,
 there cannot
exist an isometric $\Gamma$, unless $T$ is a unilateral shift.

\begin{exm}\label{cfexm4.1}
Let $U$ denote the canonical bilateral shift on $L^2(T)$; that is,
\[
 (Uf)(e^{it}) = e^{it}f(e^{it}) \quad \text{a.e.}\quad (f\in L^2(T))
\]
and let $S = U|H^2$ be the canonical unilateral shift on $H^2$. Let $V=U|L^2[(0,\pi))$ and $Q$ be the orthogonal projection of $L^2(T)=L^2([0,2\pi))$ onto $L^2([0,\pi))$. Then the following properties are immediate:
\begin{equation}\label{cfeq4.6}
 VQ = QS,\quad \ker Q =\{0\},\quad\text{and}\quad \ker Q^* =\{0\}.
\end{equation}
Now set
\begin{equation}\label{cfeq4.6a}
 T=V^*,\quad T'=S^*,\quad\text{and}\quad X=Q^*/2.\tag{4.6a}
\end{equation}
Then $T$ is unitary and $U'=U^*$ are unitary. Hence, there exists a unique intertwining lifting $Y$ if $X$ and its norm is equal to $\|X\|=1/2 < 1$.
\end{exm}

Thus even when the operators $T$ and $T'$ are very elementary, (in
this case, $T$ is a unitary operator of multiplicity one and $T'$ is
the backward shift of multiplicity one), there may not exist any
free Schur contraction that makes $\Gamma$
isometric. Again we don't see how to deduce this fact easily from
Propositions \ref{cfprop4.1} and \ref{cfprop4.2}.

In the study of the parametrization of all contractive intertwining liftings of a given intertwining contraction $X$, the case when $\|X\|<1$ is the most amenable to study. Proposition \ref{cfprop4.3} shows that in this case our  present study reduces to the case when $T$ is a unilateral shift. Related to this case we have the following .

\begin{lem}\label{cflem4.1}
 Assume $T$ is a unilateral shift, $T'$ is a $C_{\bullet 0}$-contraction (that is, $T^{*n}\to 0$ strongly) with dense range and $\|X\|<1$. If an isometric intertwining lifting $Y$ of $X$ exists, then we have
\begin{equation}\label{cfeq4.7}
 \dim \ker \bar\omega \le \dim \ker \bar\omega^*.
\end{equation}
\end{lem}

\begin{proof}
In this case the space ${\cl R}$ introduced in the proof of
Proposition \ref{cfprop4.3} is $\{0\}$, or equivalently, $U'$ is also
a unilateral shift. Consider the minimal unitary extensions
$\widehat U'\in {\cl L}(\widehat{\cl K}')$ and $\widehat T\in{\cl
L}(\widehat{\cl H})$ of $U'$ and $T$, respectively. Let $\widehat
Y\in {\cl L}(\widehat{\cl H}, \widehat{\cl K}')$ be the unique
extension (by Proposition \ref{cfprop4.3}) of $Y$ satisfying
\[
 \widehat U'\widehat Y = \widehat Y\widehat T.
\]
It is easy to see that $\widehat Y$ is isometric and thus the
multiplicities $\nu$ and $\mu$ of the bilateral shifts $\widehat U'$
and $\widehat T$, respectively, satisfy
\begin{equation}\label{cfeq4.7a}
 \mu \le \nu.\tag{4.7a}
\end{equation}
The inequality \eqref{cfeq4.7} follows directly from the equalities
\begin{align}\label{cfeq4.7b}
 &\dim \ker \bar\omega = \mu\tag{4.7b}\\
\label{cfeq4.7c}
&\dim \ker \bar\omega^* = \nu.\tag{4.7c}
\end{align}
To prove \eqref{cfeq4.7b} and \eqref{cfeq4.7c} we notice, using the
fact that $D_X$ is invertible, that
\begin{align*}
\ker\bar\omega &= D^{-1}_X \ker T^*\quad \text{and}\\
\ker\bar\omega^* &= \{D^{-1}_X X^*D_{T'}d'\oplus (d')\colon \ d'\in {\cl D}_{T'}\}.
\end{align*}
Thus
\[
 \dim \ker \bar\omega = \dim ker T^* = \mu
\]
and
\[
 \dim \ker \bar\omega^* = \dim {\cl D}_{T'} = \dim {\cl D}_{T^{\prime *}} = \nu,
\]
where the second equality follows from the fact that $\ker T^{\prime *} = \{0\}$.
\end{proof}

The preceding lemma shows that the case when the  inequality
\eqref{cfeq4.7} holds is of some interest. Therefore, \emph{through
the remaining part of this section we will assume that
\eqref{cfeq4.7} is valid}. In this case we need study only the case
when the free Schur contraction $R(z)$ is independent of $z$; that
is,  when $W(z) = W(0)$ $(z\in {\bb D})$. According to Corollary
\ref{cfcor3.3}, in this case $\Gamma(\cdot)$ is an isometry if and
only if $A_0 = \Pi W(0)$ is a $C_{0\bullet}$-contraction and $W(0)$
is an isometry. This last restriction is obviously equivalent to the
free Schur contraction $R(z)\equiv R(0)$ $(z\in {\bb D})$ being an
isometry. Therefore, if there exists such a free Schur contraction
for which the corresponding $\Gamma(\cdot)$ were not an isometry,
the operator
\begin{equation}\label{cfeq4.8}
 V = A^*_0 = W(0)^*\pi^* \in {\cl L}({\cl D})
\end{equation}
would not be a $C_{\bullet 0}$-contraction. Let  $\widehat V\in {\cl
L}(\widehat{\cl D})$ denote the minimal isometric lifting of $V$ and
let
\[
 \widehat{\cl D} = {\cl R}^\bot \oplus {\cl R}
\]
be the Wold decomposition for $\widehat V$,  where $\widehat V|{\cl
R}$ is the unitary part of $\widehat V$. Since $V$ is not a
$C_{\bullet 0}$-contraction, there exists an $r_0\in {\cl R}$
satisfying $d_0 = Pr_0\ne 0$, where $P$ denotes the orthogonal
projection of $\widehat{\cl D}$ onto ${\cl D}$. Let
\begin{equation}\label{cfeq4.8a}
 d_n = P\widehat V^{*n} r_0\qquad (n=0,1,\ldots).\tag{4.8a}
\end{equation}
Then
\begin{align}\label{cfeq4.8b}
 &Vd_{n+1} = P\widehat V\widehat V^{*n+1}r_0 = P\widehat V^{*n}r_0 = d_n,\tag{4.8b}\\
\label{cfeq4.8c}
&0 \le \|d_0\| \le \|d_1\| \le\cdots\le \|d_n\|\le\cdots,\tag{4.8c}
\end{align}
and
\begin{equation}\label{cfeq4.8d}
 \|d_n\|\le \|r_0\|\qquad (n=0,1,2,\ldots).\tag{4.8d}
\end{equation}

At this moment it is worth noticing that  we have actually proven part
of the following characterization of a contraction which is not a
$C_{\bullet 0}$-contraction, a fact which may be useful elsewhere.

\begin{lem}\label{cflem4.2}
 Let $T\in {\cl L}({\cl H})$ be a contraction. Then $T$ is not a $C_{\bullet 0}$-contraction
 if and only if there exists a
 bounded sequence $\{h_n\}^\infty_{n=0}\subset {\cl H}$ such that
\begin{equation}\label{cfeq4.9}
 h_0\ne 0,\quad h_n= Th_{n+1} \qquad (n=0,1,\ldots).
\end{equation}
\end{lem}

\begin{proof}
 It remains to prove that if such  a sequence exists then $T$ is not a $C_{\bullet 0}$-contraction. To this end note that
\[
\|h_0\| \le \|h_1\| \le\cdots\le \|h_n\| \le \|h_{n+1}\| \le\cdots\le M<\infty,
\]
where $M$ is the supremum in \eqref{cfeq4.9}. Choose $n_0$ large enough for $\|h_{n_0}\|\ge \sqrt{15}~M/8$ to hold. Then for any $N=0,1,\ldots$, we have
\begin{align*}
\|(I-T^{*N}T^N)h_{{n_0}+N}\|^4 &\le \|(I-T^{*N}T^N)^{1/2} h_{n_0+N}\|^2 \|h_{n_0+N}\|^2\\
&\le ((I-T^{*N}T^N)h_{n_0+N}, h_{n_0+N})M^2 = (\|h_{n_0+N}\|^2 - \|h_{n_0}\|^2)M^2\\
&\le M^4/16.
\end{align*}
Hence
\begin{align*}
\|T^Nh_{n_0}\| &\ge \|h_{n_0+N}\| - \|(I-T^{*N}T^N)h_{n_0+N}\| \\
&\ge \sqrt{15}~M/8 - M/4>0
\end{align*}
for all $N=0,1,\ldots$~. This proves that $T$ is not a $C_{\bullet 0}$-contraction.
\end{proof}

We return now to our particular considerations. The relation \eqref{cfeq4.8b} can be written as
\begin{equation}\label{cfeq4.10}
 W(0)^*\Pi^*d_{n+1} = d_n\qquad (n=0,1,\ldots).
\end{equation}
Applying  $\bar\omega$ to these last equalities we obtain
\begin{equation}\label{cfeq4.10a}
\bar\omega\bar\omega^* \Pi^* d_{n+1} = \bar\omega d_n\qquad (n=0,1,\ldots).\tag{4.10a}
\end{equation}
Note that \eqref{cfeq4.10} also implies
\begin{equation}\label{cfeq4.10b}
 \|d_0\|\le \|d_1\|\le \|d_2\|\le\ldots~.\tag{4.10b}
\end{equation}
Thus we obtain the following.

\begin{lem}\label{cflem4.3}
 Let $\omega$ have $($besides \eqref{cfeq4.7}$)$ the following property:
\begin{itemize}
 \item[\rm (a)] Any sequence $\{d_n\}^\infty_{n=0} \subset {\cl D}$ for which \eqref{cfeq4.10a} and \eqref{cfeq4.10b}  are valid is either identically zero or unbounded.
\end{itemize}
Then for any isometric free Schur contraction $R(z) \equiv R(0)$ $(z\in{\bb D})$,  the corresponding operator $\Gamma(\cdot)$ is also an isometry.
\end{lem}

This lemma does not preclude the possibility that the conclusion of Lemma \ref{cflem4.3} holds under a weaker condition than the condition (a).

Indeed, let us assume that we have a sequence $\{d_n\}^\infty_{n=0} \subset {\cl D}$ satisfying the condition \eqref{cfeq4.10}. We extend recursively the definition of the $d_n$'s as follows:
\begin{equation}\label{cfeq4.11}
 d_{n-1} = W(0)^* \Pi^* d_n
\end{equation}
for $n=0, n=-1,\ldots$~. Let ${\cl D}_0$ be the linear space spanned by $\{d_n\}^\infty_{n=-\infty}$. Then the linear map $C$ defined from $(I-\bar\omega\bar\omega^*)\Pi^*{\cl D}_0$ into $(I-\bar\omega^*\bar\omega){\cl D}_0$ by
\begin{equation}\label{cfeq4.11a}
 C(I-\bar\omega\bar\omega^*)\Pi^*d_{n+1} = (I-\bar\omega^*\bar\omega)d_n\qquad (n\in {\bb Z})\tag{4.11a}
\end{equation}
extends  by continuity to $\ovl C = R(0)^*|((I-\bar\omega\bar\omega^*)\Pi^*{\cl D}_0)^-$. Clearly, $\ovl C$ \emph{is a contraction and its definition depends only on $\omega$ and the sequence $\{d_n\}^\infty_{n=0}$ satisfying} \eqref{cfeq4.10}.   Moreover, by its construction $\ovl C$ extends to a co-isometry (namely $R(0)^*$) from $\ker \bar\omega^*$ onto $\ker \bar\omega$.

 To continue our analysis we now need the following.

\begin{lem}\label{cflem4.4}
Let ${\cl H}$ and ${\cl H}'$ be two Hilbert spaces with subspaces
${\cl M}\subset {\cl H}$ and ${\cl M}'\subset {\cl H}'$. Let $C\in
{\cl L}({\cl M}', {\cl M})$ be a contraction with dense range in
${\cl M}$. Then $C$ has a coisometric extension $\widehat{C}\in {\cl
L}({\cl H}', {\cl H})$ if and only if
\begin{equation}\label{cfeq4.12}
\dim({\cl H}' \ominus {\cl M}') \ge \dim(({\cl H}\ominus {\cl M}) \oplus {\cl D}_{C^*}).
\end{equation}
\end{lem}

\begin{proof}
If a coisometric extension $\widehat C$ of $C$ exists, then for $h\in {\cl H} \ominus {\cl M}$ we have
\[
(\widehat C^*h,m') = (h,Cm') = 0\qquad (m'\in {\cl M}')
\]
and so
\begin{equation}\label{cfeq4.12a}
\widehat C^*({\cl H}\ominus {\cl M}) \subset {\cl H}\ominus {\cl M}'.\tag{4.12a}
\end{equation}
Clearly, we also have
\begin{equation}\label{cfeq4.12b}
\widehat C^*{\cl M}\perp \widehat C^*({\cl H}\ominus {\cl M}),\tag{4.12b}
\end{equation}
and
\begin{equation}\label{cfeq4.12c}
P'_{{\cl M}'}\widehat C^*|{\cl M} = C^*, \text{ where $P'_{\cl M}$ is the orthogonal projection of ${\cl H}'$ onto } {\cl M}'.\tag{4.12c}
\end{equation}
Thus
\begin{equation}\label{cfeq4.12d}
\widehat C^*m = C^*m + XD_{C^*}m\qquad (m\in {\cl M}),\tag{4.12d}
\end{equation}
where $X \in {\cl L}({\cl D}_{C^*}, {\cl H}'\ominus {\cl M}')$ is an isometry. Due to \eqref{cfeq4.12b} we have
\[
X{\cl D}_{C^*}\perp \widehat C^*({\cl H}\ominus {\cl M})
\]
and therefore \eqref{cfeq4.12} holds. Conversely, if \eqref{cfeq4.12} holds we can define an isometric operator $C_1$ from ${\cl H}$ into ${\cl H}'$ in the following way. First, due to \eqref{cfeq4.12} we can find two mutually orthogonal subspaces ${\cl X}$  and ${\cl Y}$ of ${\cl H}'\ominus {\cl M}'$ such that
\[
 \dim {\cl X} = \dim {\cl D}_C, \quad \dim {\cl Y}  =\dim {\cl H}\ominus {\cl M}.
\]
Choose for $C_1|{\cl Y}$ any unitary operator $\in {\cl L}({\cl H}\ominus {\cl M}, {\cl Y})$ and define $C_1|{\cl M}$ by
\begin{equation}\label{cfeq4.12e}
 C_1m = C^*m + XD_{C^*}m \qquad (m\in{\cl M}),\tag{4.12e}
\end{equation}
where $X$ is any unitary operator in ${\cl L}({\cl D}_{C^*},{\cl X})$. The operator thus defined on ${\cl H}$ is isometric and
\[
 P'_{{\cl M}'}C_1|{\cl M} = C^*.
\]
Consequently, $C^*_1$ is coisometric and
\begin{align*}
 (C^*_1m',h) &= (m',C_1h) = (m', C_1P_{\cl M}h) =\\
&= (m',C^*P_{\cl M}h) = (Cm'_1P_{\cl M}h) = Cm',h)
\end{align*}
for all $m'\in {\cl M}'$, $h\in {\cl H}$, and hence $C^*_1|{\cl M}' =C$, where $P_{\cl M}$ is the orthogonal projection of ${\cl H}$ onto ${\cl M}$.$\hfill\square$

Returning to the discussion preceding the above lemma, we deduce that the contraction $\ovl C$ must satisfy the condition
\begin{equation}\label{cfeq4.13}
\dim(\ker \bar\omega^*)\ge \dim(\ker\bar\omega \oplus {\cl D}_{C^*}).
\end{equation}
Thus we have proved the following.
\end{proof}

\begin{prop}\label{cfprop4.4}
Assume that there exists a not identically zero sequence $\{d_n\}^\infty_{n=0} \subset {\cl D}$ satisfying \eqref{cfeq4.10a} and \eqref{cfeq4.10b}. In order that this sequence also satisfies \eqref{cfeq4.10} for an appropriate free Schur contraction $R(z) = R(0)$ $(z\in {\bb D})$ when $R(0)$ is isometric, the following set of properties is necessary and sufficient:
\begin{itemize}
 \item[\rm (a)] The sequence $\{d_n\}^\infty_{n=0}$ can be extended to a bilateral sequence $\{d_n\}^\infty_{n=-\infty}$ satisfying
\begin{equation}\label{cfeq4.14}
\bar\omega\bar\omega^* \Pi^*d_{n+1} = \bar\omega d_n\qquad (n\in {\bb Z});
\end{equation}
\item[\rm (b)] the definition \eqref{cfeq4.11a} yields, by linearity and continuity, a contraction in ${\cl L}(((I-\bar\omega\bar\omega^*)\Pi{\cl D}_0)^-$, $((I-\bar\omega^*\bar\omega){\cl D}_0)^-)$, where ${\cl D}_0$ is the linear span of $\{d_n\}^\infty_{n=-\infty}$;
\item[\rm (c)] the inequality
\end{itemize}
\begin{equation}\label{cfeq4.14a}
 \dim(\ker \bar\omega^*) \ge \dim(\ker \bar\omega \oplus {\cl D}_{C^*})\tag{4.14a}
\end{equation}
\begin{itemize}
\item[] holds.
\end{itemize}
\end{prop}

Note that \eqref{cfeq4.14a} is a more stringent condition than \eqref{cfeq4.7}.

Finally, the proof of Lemma \ref{cflem4.3} allows us to infer the following complement to Proposition \ref{cfprop4.4} and Lemma \ref{cflem4.3}.

\begin{prop}\label{cfprop4.5}
Let $\{d_n\}$ be a sequence satisfying \eqref{cfeq4.10a}, \eqref{cfeq4.10b}, all the properties (a), (b), and (c)  in Proposition \ref{cfprop4.4} and
\begin{equation}\label{cfeq4.14b}
 \sup_{n\ge 0} \|d_n\| < \infty.\tag{4.14b}
\end{equation}
Then no operator $\Gamma(\cdot)$ corresponding to a free Schur contraction provided by Proposition \ref{cfprop4.4} is isometric.
\end{prop}

We conclude this note with a closer look at the case
\begin{equation}\label{cfeq4.15}
 \|X\|<1,
\end{equation}
in which the partial isometries $\bar\omega$ and $\bar\omega^*$ can be given in an explicit form. Indeed, in this case $D_X$ and $T^*D^2_XT$ are invertible operators in ${\cl H}$ and ${\cl D} = {\cl D}_X = {\cl H}$,
\begin{align}\label{cfeq4.15a}
\bar\omega^*\bar\omega &= D_XT(T^*D^2_XT)^{-1}T^*D_X,\tag{4.15a}\\
\label{cfeq4.15b}
\bar\omega &= \begin{bmatrix}
               D_X\\ D_{T'}X
              \end{bmatrix} (T^*D^2_XT)^{-1}T^*D_X, \quad \text{and}\tag{4.15b}\\
\label{cfeq4.15c}
\bar\omega^* &= D_XT(T^*D^2_XT)^{-1}[D_XX^*D_{T'}].\tag{4.15c}
\end{align}
With this preparation we can now prove the following result.

\begin{prop}\label{cfprop4.6}
Assume $T$ is a unilateral shift (of any multiplicity) and that the relations \eqref{cfeq4.7} and \eqref{cfeq4.15} are satisfied. Let $R_0\in {\cl L}(\ker\bar\omega, \ker\bar\omega^*)$ be any isometry. Define the free Schur contraction by $R(z)=R_0$ $(z\in{\bb D})$ and let $Y$ be the corresponding intertwining lifting of $X$. Then $Y$ is an isometry.
\end{prop}

\begin{proof}
It will be sufficient to prove that in the present case Property (a) in Lemma \ref{cflem4.3} is satisfied. So, let $\{d_n\}^\infty_{n=0}$ be a sequence in ${\cl H}$ satisfying the relations \eqref{cfeq4.10a} and \eqref{cfeq4.10b}. Applying $\bar\omega^*$ on both sides of identity \eqref{cfeq4.10a} we obtain
\begin{equation}\label{cfeq4.16}
\bar\omega^*\Pi^*d_{n+1} = \bar\omega^*\bar\omega d_n\qquad (n=0,1,2,\ldots,).
\end{equation}
Introducing in \eqref{cfeq4.16} the explicit forms \eqref{cfeq4.15a} of $\bar\omega^*\bar\omega$; and \eqref{cfeq4.15c} of $\bar\omega^*$, respectively, we obtain
\[
 D_XT(T^*D^2_XT)^{-1}D_X d_{n+1} = D_XT(T^*D^2_XT)^{-1} T^*D_X d_n
\]
$(n=0,1,\ldots)$, and whence
\[
 d_{n+1} = D^{-1}_XT^*D_X d_n\qquad (n=0,1,2,\ldots).
\]
We infer
\[
d_n = D^{-1}_XT^{*n}D_Xd_0\qquad (n=0,1,2,\ldots),
\]
where $T^{*n}\to 0$ strongly. This together with \eqref{cfeq4.10b} forces $d_n=0$ for all $n\ge 0$.
\end{proof}

\begin{rk}\label{cfrk4.1}
Under the assumptions of Proposition \ref{cfprop4.6}, the inequality \eqref{cfeq4.7} obtains an explicit form. Indeed, since
\begin{align*}
\ker \bar\omega^*\bar\omega &= D^{-1}_X \ker T^*\\
\intertext{and}
\ker \bar\omega\bar\omega^* &= \left\{\begin{bmatrix}
                                       D^{-1}_XX^*D_Td'\\ d'
                                      \end{bmatrix}\colon \ d'\in {\cl D}_{T'} (={\cl D}')\right\},
\end{align*}
we have
\begin{equation}\label{cfeq4.17}
 \dim \ker \bar\omega^*\bar\omega = \dim \ker T^* = \dim {\cl D}_{T^*},
\end{equation}
and
\begin{equation}\label{cfeq4.17a}
 \dim \ker \bar\omega \bar\omega^* = \dim {\cl D}_{T'}.\tag{4.17a}
\end{equation}
Consequently, introducing \eqref{cfeq4.17} and \eqref{cfeq4.17a} in \eqref{cfeq4.7}, the last relation takes the form
\begin{equation}\label{cfeq4.17b}
 \dim {\cl D}_{T'} \ge \dim {\cl D}_{T^*}.\tag{4.17b}
\end{equation}
\end{rk}

In fact, the previous proof can be modified to yield the following slight improvement of Proposition \ref{cfprop4.6}.

\begin{prop}\label{cfprop4.7}
Assume that $T$ is a unilateral shift and \eqref{cfeq4.7} is satisfied.
 Assume also that
\begin{equation}\label{cfeq4.18}
D_X \text{ and } D_XT \text{ both have closed range.}
\end{equation}
Then the conclusion of Proposition \ref{cfprop4.6} is valid.
\end{prop}

\begin{proof}
We will use the proof of Proposition \ref{cfprop4.6} replacing the inverse for $T^*D^2_XT$ by a left inverse which the fact that $D_XT$ and $D_X$ have closed range will allow us to define. The key here is to show that the range of  $T^*D_X$ is contained in the support of $T^*D^2_XT$.
\end{proof}

\begin{rk}\label{cfrk4.2}
One can verify using the definition that $\ker\bar\omega = D_X{\cl H}\cap D_{T^*}{\cl H}$ and $\ker\bar\omega^* = D_{T'}{\cl H}' \cap D_{X^*} {\cl H}'$. Therefore, in the context of Proposition \ref{cfprop4.7}, \eqref{cfeq4.7} becomes
\begin{equation}\label{cfeq4.17c}
\dim(D_X{\cl H}\cap D_{T^*}{\cl H}) \le \dim(D_{T'}{\cl H}'\cap D_{X^*}{\cl H}').\tag{4.17c}
\end{equation}
Obviously, \eqref{cfeq4.17c} reduces to \eqref{cfeq4.17b} if $\|X\|<1$ and here $D_X$ is invertible.
\end{rk}

\begin{rk}\label{cfrk4.3}
We begin by noting that \eqref{cfeq4.18} in Proposition \ref{cfprop4.7} is equivalent to the assumption that $D_X$ and $D_{XT}$ have closed range.
Proposition \ref{cfprop4.7} has a direct consequence concerning an apparently more general setting of the Commutant Lifting Theorem. Indeed, if $T_0\in {\cl L}({\cl H}_0)$ is a contraction, $X_0\in {\cl L}({\cl H}_0, {\cl H}')$ satisfies
\begin{equation}\label{cfeq4.19}
 X_0T_0 = T'X_0, \text{ and both $D_{X_0}$ and $D_{X_0T_0}$ have closed range.}
\end{equation}
If $T\in {\cl L}({\cl H})$ is the minimal isometric lifting of $T_0$, then
\begin{equation}\label{cfeq4.19a}
 X = X_0P_0,\tag{4.19a}
\end{equation}
where $P_0$ is the orthogonal projection of ${\cl H}$ onto ${\cl H}_0$, will satisfy
\begin{equation}\label{cfeq4.19b}
 XT = T'X \text{ and both $D_X$ and $D_{XT}$ will have closed range.}\tag{4.19b}
\end{equation}

Let $Y$ be any contractive intertwining lifting of $X$. Then  the dimensions satisfy
\begin{equation}\label{cfeq4.19c}
P'Y = X_0P_0\quad \text{and}\quad YT=T'Y.\tag{4.19c}
\end{equation}
Moreover, any contraction $Y\in {\cl L}({\cl H}, {\cl K}')$ satisfying \eqref{cfeq4.19b} (actually referred to as a contractive intertwining lifting of $X_0$) will be a contractive intertwining lifting of $X$. Now assume that $T_0$ is a $C_{\bullet 0}$-contraction. This implies that $T$ is a shift and that
\[
 \dim {\cl D}_{T^*} = \dim {\cl D}_{T^*_0}.
\]
(see \cite[Ch.~II]{SzNF2}). Thus, if
\begin{equation}\label{cfeq4.20}
 \dim({\cl D}_{X_0}{\cl H}_0\cap {\cl D}_{T_0}{\cl H}_0) \le \dim({\cl D}_{T'}{\cl H}'\cap {\cl D}_{X^*_0}{\cl H}'),
\end{equation}
we can apply Proposition \ref{cfprop4.7} to the present setting and conclude \emph{that the set of the isometric intertwining liftings of $X_0$ is not empty and, moreover, that for every free Schur contraction of the form $R(z) = R_0$ $(z\in {\bb D})$ with $R_0$ an isometry, the corresponding contractive intertwining lifting $Y$ of $X_0$ is also an isometry}; in connection with this result see \cite{FFT,LT}.
\end{rk}

\bigskip

\n Department of Mathematics, Indiana University,\\
Bloomington, Indiana \  47405\\
Email address: bercovic@indiana.edu.
\medskip 

\n Department of Mathematics, Texas  A\&M University,\\
College Station, Texas \ 77843\\
Email address: rdouglas@math.tamu.edu\medskip

\n Department of Mathematics, Texas A\&M University,\\
College Station, Texas \ 77843\\
Email address: foias@math.tamu.edu
\end{document}